\documentclass[
submission
]{dmtcs-episciences}

\usepackage[utf8]{inputenc}
\usepackage{subfigure}
\newtheorem{theorem}{Theorem}[section]
\newtheorem{lemma}[theorem]{Lemma}

\newtheorem{corollary}[theorem]{Corollary}

\usepackage[round]{natbib}

\author{Tianlong Ma\affiliationmark{1}
  \and Yaping Mao\affiliationmark{2,4}\footnote{Corresponding author}\thanks{Supported by the National Science Foundation of China
(Nos. 11601254, 11661068, 61763041 and 11551001) and the Science
Found of Qinghai Province (No. 2019-ZJ-921)}
  \and Eddie Cheng\affiliationmark{3}
  \and Christopher Melekian\affiliationmark{3} }
\title[Fractional matching preclusion for generalized
  augmented cubes]{Fractional matching preclusion for generalized\\
  augmented cubes}
\affiliation{
  Department of Basic Research, Qinghai University, Xining, Qinghai 810016, China\\
  School of Mathematics and Statistics, Qinghai Normal University, Xining, Qinghai 810008, China\\
  Department of Mathematics and Statistics, Oakland University, Rochester, MI USA 48309\\
  Academy of Plateau Science and Sustainability, Xining, Qinghai 810008, China}
\received{2019-1-11}

\revised{2019-5-2}

\accepted{2019-7-21}

\begin{document}
\publicationdetails{21}{2019}{4}{6}{5074}
\maketitle
\begin{abstract}
  The \emph{matching preclusion number} of a graph is the minimum
number of edges whose deletion results in a graph that has neither
perfect matchings nor almost perfect matchings. As a generalization,
Liu and Liu (2017)
 recently introduced the concept of
fractional matching preclusion number. The \emph{fractional matching
preclusion number} of $G$ is the minimum
number of edges whose deletion leaves the resulting graph without a
fractional perfect matching. The \emph{fractional strong matching
preclusion number} of $G$ is the minimum
number of vertices and edges whose deletion leaves the resulting
graph without a fractional perfect matching. In this paper, we obtain the fractional matching
preclusion number and the fractional strong matching preclusion number for generalized augmented cubes. In addition, all the optimal fractional strong matching preclusion sets of these graphs are categorized.
\end{abstract}

\section{Introduction}
\label{sec:in}

\label{sec:in}
Parallel computing is an important area of computer science and engineering. The underlying
topology of such a parallel machine or a computer network is the interconnection network. Computing nodes are processors where the resulting system is a multiprocessor supercomputer, or
they can be computers in which the resulting system is a computer network. It is unclear where the computing future is headed. It may lead to more research in multiprocessor supercomputers, physical networks or networks in the cloud. Nevertheless, the analysis of such networks will always be important. One important aspect of network analysis is fault analysis, that is, the study of how faulty processors/links will affect the structural properties of the underlying interconnection networks, or simply graphs.

All graphs considered in this paper are undirected, finite and
simple. We refer to the book \cite{Bondy} for graph theoretical
notations and terminology not defined here. For a graph $G$, let
$V(G)$, $E(G)$, and $(u,v)$ ($uv$ for short) denote the set of
vertices, the set of edges, and the edge whose end vertices are $u$ and $v$,
respectively. We use $G-F$ to denote the subgraph of $G$
obtained by removing all the vertices and (or) the edges of $F$. We denote by $C_n$ the cycle with $n$ vertices. A cycle (respectively, path) in $G$ is called a Hamiltonian cycle (respectively, Hamiltonian path)
if it contains every vertex of $G$ exactly once. We divide our introduction into the following three
subsections to state the motivations and our results of this paper.

\subsection{(Strong) matching preclusion number}

A \emph{perfect matching} in a graph is a set of edges such that
each vertex is incident to exactly one of them, and an
\emph{almost-perfect matching} is a set of edges such that each
vertex, except one, is incident to exactly one edge in the set, and the
exceptional vertex is incident to none. A graph with an even number of vertices is an \emph{even graph};
otherwise it is an \emph{odd graph}. Clearly an even graph cannot have an almost perfect matching
and an odd graph cannot have a perfect matching. A \emph{matching preclusion
set} of a graph $G$ is a set of edges whose deletion leaves $G$ with
neither perfect matchings nor almost-perfect matchings, and the
\emph{matching preclusion number} of a graph $G$, denoted by $mp(G)$
is the size of a smallest matching preclusion set  of $G$.
The concept of matching preclusion was introduced by
\cite{brigham2005perfect} as a measure of robustness of interconnection
networks in the event of edge failure. It also
has connections to a number of other theoretical topics, including
conditional connectivity and extremal graph theory. We refer the
readers to \cite{cheng2007matching, cheng2009conditional, jwo1993new,li2016matching,mao2018strong, wang2019matching} for
further details and additional references.

A matching preclusion set of minimum cardinality is called
\emph{optimal}. For graphs with an even number of vertices, one can see the set of edge incident
to a single vertex is a matching preclusion set; such a set is called a \emph{trivial matching preclusion
set}. A graph $G$ satisfying $mp(G) = \delta(G)$ is said to be
\emph{maximally matched}, and in a maximally matched graph some
trivial matching preclusion set is optimal. Furthermore, a graph $G$
is said to be \emph{super matched} if every optimal matching
preclusion set is trivial. Immediately we see that every super
matched graph is maximally matched. Being super matched is a
desirable property for any real-life networks, as it is unlikely
that in the event of random edge failure, all of the failed edges
will be incident to a single vertex. (Here one can think of vertices as processors in a parallel machines and edges as physical links.)

A set $F$ of edges and vertices of $G$ is a \emph{strong matching
preclusion set} (SMP set for short) if $G-F$ has neither perfect
matchings nor almost-perfect matchings. The \emph{strong matching
preclusion number} (SMP number for short) of $G$, denoted by
$smp(G)$, is the minimum size of SMP sets of $G$. An SMP set is
optimal if $|F|=smp(G)$. The problem of strong matching preclusion
set was proposed by \cite{park2011strong}.
We remark that if $F$ is an optimal strong matching preclusion set, then we may assume that no edge in $F$ is incident to a vertex in $F$.
According to the
definition of $mp(G)$ and $smp(G)$, we have that $smp(G)\leq mp(G)\leq \delta(G)$. We say a graph is \emph{strongly maximally matched} if
$smp(G)=\delta(G)$. If $G-F$ has isolated vertices and $F$ is an optimal strong matching preclusion set, then $F$ is \emph{basic}. If, in addition, $G$
is even and  $F$ has an even number of vertices, then $F$ is \emph{trivial}. A strongly maximally matched even graph is \emph{strongly super matched}
if every optimal strong matching preclusion set is trivial.

\subsection{Fractional (strong) matching
preclusion number}

A standard way to consider matchings in polyhedral combinatorics is as follows. Given a set of edges $M$ of $G$, we define $f^M$ to be the indicator function of $M$, that is, $f^M:E(G)\longrightarrow \{0,1\}$ such that
$f^M(e)=1$ if and only if $e\in M$. Let $X$ be a set of vertices of $G$. We denote $\delta'(X)$ to be the set of edges with exactly one end in $X$. If $X=\{v\}$, we write $\delta'(v)$ instead of $\delta'(\{v\})$. (We remark that it
is common to use $\delta(X)$ is the literature. However, since it is also common to use $\delta(G)$ to denote the minimum degree of vertices in $G$. Thus we choose to use $\delta'$ for this purpose.)
Thus $f^M:E(G)\longrightarrow \{0,1\}$ is the indicator function of the perfect matching $M$ if
$\sum_{e\in\delta'(v)} f^M(e)=1$ for each vertex $v$ of $G$. If we replace ``$=$'' by ``$\leq$,'' then $M$ is a \emph{matching} of $G$. Now
$f^M:E(G)\longrightarrow \{0,1\}$ is the indicator function of the almost perfect matching $M$ if
$\sum_{e\in\delta'(v)} f^M(e)=1$ for each vertex $v$ of $G$, except one vertex say $v'$, and $\sum_{e\in\delta'(v')} f^M(e)=0$.
It is also common to use $f(X)$ to denote $\sum_{x\in X} f(x)$.
We note that it follows from the definition that $f^M(E(G))=\sum_{e\in E(G)} f^M(e)$ is $|M|$ for a matching $M$. In particular, $f^M(E(G))=|V(G)|/2$ if $M$ is a perfect matching and
 $f^M(E(G))=(|V(G)-1)|/2$ if $M$ is an almost perfect matching.

A standard relaxation from an integer setting to a continuous setting is to replace the codomain of the indicator function from $\{0,1\}$ to the interval $[0,1]$. Let $f:E(G)\longrightarrow [0,1]$. Then $f$ is a \emph{fractional matching}
if $\sum_{e\in\delta^{'}(v)} f(e)\leq 1$ for each vertex $v$ of $G$; $f$ is a \emph{fractional perfect matching}
if $\sum_{e\in\delta^{'}(v)} f(e)=1$ for each vertex $v$ of $G$; and $f$ is an \emph{fractional almost perfect matching}
if $\sum_{e\in\delta^{'}(v)} f(e)=1$ for each vertex $v$ of $G$ except one vertex say $v'$, and $\sum_{e\in\delta^{'}(v^{'})} f(e)=0$. We note that if $f$ is a fractional perfect matching, then
\[ f(E(G))=\sum_{e\in E(G)} f(e)=\frac{1}{2}\sum_{v\in V(G)}\sum_{e\in \delta^{'}(v)} f(e)=\frac{|V(G)|}{2}; \]
and
if $f$ is a fractional almost perfect matching, then
\[ f(E(G))=\sum_{e\in E(G)} f(e)=\frac{1}{2}\sum_{v\in V(G)}\sum_{e\in \delta^{'}(v)} f(e)=\frac{|V(G)|-1}{2}. \]

We note that although an even graph cannot have an almost perfect matching, an even graph can have a fractional almost perfect matching. For example, let $G$ be the graph with two components, one with a $K_3$ and one with a $K_1$.
Now assign every edge a $1/2$, then the corresponding indicator function is a fractional almost perfect matching. Similarly, an odd graph can have a fractional perfect matching. Thus to generalize the concept of matching preclusion sets,
there are choices. In particular, should we preclude fractional perfect matchings only, or both fractional perfect matchings and fractional almost perfect matchings.
Recently, \cite{liu2017fractional} gave one such generalization. An edge subset $F$ of $G$ is a \emph{fractional matching preclusion
set} (FMP set for short) if $G-F$ has no fractional perfect
matchings. In addition, the \emph{fractional matching preclusion number} (FMP
number for short) of $G$, denoted by $fmp(G)$, is the minimum size
of FMP sets of $G$. So their choice was to preclude fractional perfect matchings only.

Let $G$ be an even graph. Suppose $F$ is an FMP set. Then $G-F$ has no fractional perfect matchings. In particular, $G-F$ has no perfect matchings. Thus $F$ is a matching preclusion set. Hence
\[ mp(G)\leq fmp(G). \]

 As pointed out in
\cite{liu2017fractional}, this inequality does not hold if $G$ is an odd graph. The reason is due to the definition. Here for the integer case, one precludes almost perfect matchings whereas for the fractional case, one precludes fractional perfect matchings.
So there is a mismatch. If one were to preclude perfect matchings even for the integer case, then the preclusion number is 0 and the inequality will holds. This is a minor point as in application to interconnection networks, only even graphs will
be considered. For the rest of the paper, we only consider even graphs. Since a graph with an isolated vertex cannot have fractional perfect matchings, we have $fmp(G)\leq\delta(G)$. Thus if $G$ is even, we have the following inequalities
\[ mp(G)\leq fmp(G)\leq \delta(G). \]
Therefore, if $G$ is maximally matched, then $fmp(G)=\delta(G)$.

\cite{liu2017fractional} also gave a generalization of strong matching preclusion.
A set $F$ of edges and vertices of $G$ is a \emph{fractional strong
matching preclusion set} (FSMP set for short) if $G-F$ has no
fractional perfect matchings. The \emph{fractional strong matching
preclusion number} (FSMP number for short) of $G$, denoted by
$fsmp(G)$, is the minimum size of FSMP sets of $G$. Again the fractional version preclude fractional perfect matchings only.  Since a fractional matching preclusion set is a fractional strong matching preclusion set, it is clear that
$$
fsmp(G)\leq fmp(G)\leq \delta(G).
$$
An FMP set $F$ is optimal if $|F|= fmp(G)$. If $fmp(G)=\delta(G)$, then $G$ is \emph{fractional maximally matched}; if, in addition, $G-F$ has isolated vertices for every optimal fractional matching
preclusion set $F$, then $G$ is \emph{fractional super matched}. An FSMP set $F$ is optimal if $|F|= fsmp(G)$.
If $fsmp(G)=\delta(G)$, then $G$ is \emph{fractional strongly maximally matched}; if, in addition, $G-F$ has an isolated vertices for every optimal fractional strong matching
preclusion set $F$, then $G$ is \emph{fractional strongly super matched}.
\subsection{Variants of Hypercubes}

The class of hypercubes is the most basic class of interconnection networks. However, hypercubes have shortcomings including embedding issues. A number of
variants were introduced to address some of these issues, and one popular variant is the class of augmented cubes
given by \cite{choudum2002augmented}. By design, the augmented cube graphs are superior in
many aspects. They retain many important properties of hypercubes and they possess some embedding
properties that the hypercubes do not have. For instance, an augmented cube of the $n$th dimension contains cycles of all lengths
from $3$ to $2^n$ whereas the hypercube contains only even cycles. As shown in \cite{park2011strong}, bipartite graphs are poor interconnection
networks with respect to the strong matching preclusion property. However, augmented cubes have good strong matching preclusion properties as shown in \cite{cheng2010matching}.

We now define the $n$-dimensional augmented cube $AQ_n$ as follows. Let $n\geq 1$, the graph $AQ_n$ has $2^n$ vertices, each labeled
by an $n$-bit $\{0,1\}$-string $u_1u_2\cdots u_n$. Then $AQ_1$ is isomorphic to the complete graph $K_2$ where one
vertex is labeled by the digit $0$ and the other by $1$. For $n\geq 2$, $AQ_n$ is defined recursively by using two copies of $(n-1)$-
dimensional augmented cubes with edges between them. We first add the digit $0$ to the beginning of the binary strings of all
vertices in one copy of $AQ_{n-1}$, which will be denoted by $AQ^0_{n-1}$, and add the digit $1$ to the beginning of all the vertices of the
second copy, which will be denoted by $AQ^1_{n-1}$.
We call simply $AQ^0_{n-1}$ and $AQ^1_{n-1}$ two copies of $AQ_{n}$. We now describe the edges between these two copies. Let $u=0u_1u_2\cdots u_{n-1}$
and $v=1v_1v_2\cdots v_{n-1}$ be vertices in $AQ^0_{n-1}$ and $AQ^1_{n-1}$, respectively. Then $u$ and $v$ are adjacent if and only if one of the
following conditions holds:
\begin{itemize}
\item[] (1) $u_i=v_i$ for every $i\geq 1$. In this case, we call the edge $(u, v)$ a \emph{cross edge} of $AQ_{n}$ and say $u=v^x$ and $v=u^x$.

\item[] (2) $u_i\neq v_i$ for every $i\geq 1$. In this case, we call $(u, v)$ a \emph{complement edge} of
$AQ_{n}$ and denote $u=v^c$ and $v=u^c$. (Here we use the notation $v^c$ to means the complement of $v$, that is every 0 becomes
a 1 and every 1 becomes a 0.)
\end{itemize}

Clearly $AQ_n$ is $(2n-1)$-regular and it is known that $AQ_n$ is vertex transitive. Another important fact is that the connectivity of $AQ_n$ is $2n-1$ for $n\geq 4$. Some recent papers on augmented cubes include
\cite{angjeli2013linearly,chang2010conditional,cheng2010matching,cheng2013strong,hsieh2007cycle,hsieh2010conditional,ma2007panconnectivity,ma2008super}. A few examples of augmented cubes are shown in Fig. 1. We note that without the complement edges, it coincides with the recursive definition of hypercubes. We note that
a non-recursive classification of a complement edge $(u,v)$ is $u=ab$ and $v=ab^c$ where $a$ is a (possibly empty) binary string and $b$ is an non-empty binary string. (Here $ab$ is the usual concatenation notation of $a$
and $b$.)

In fact, augmented cubes can be further generalized. The cross edges and complement edges are edge disjoint perfect matchings and they can be replaced by other edges. We define the set ${\cal GAQ}_4=\{AQ_4\}$. For $n\geq 5$, ${\cal GAQ}_n$ consists of all graphs that can be obtained in the following way: Let $G_1,G_2\in {\cal GAQ}_{n-1}$, where $G_1=(V_1,E_1)$ and $G_2=(V_2,E_2)$ may not be distinct; construct the graph $(V_1\cup V_2,E_1\cup E_2\cup M_1\cup M_2)$ where $M_1$ and $M_2$
are edge disjoint perfect matchings between $V_1$ and $V_2$. It follows from the definition that if $G\in {\cal GAQ}_n$, then $G$ is a $(2n-1)$-regular graph on $2^n$ vertices.
These are the \emph{generalized augmented cubes}. In this paper, we study the fractional strong matching preclusion problem for these graphs.

\begin{figure}[htbp]
  \begin{center}
    \includegraphics[scale=0.9]{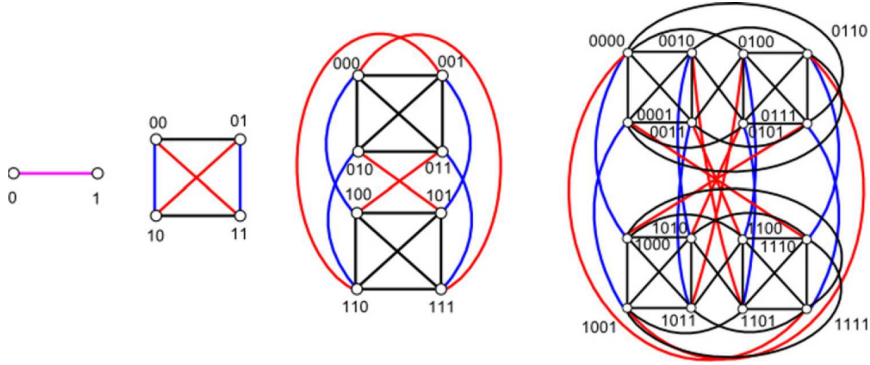}
    \caption{ Augmented cubes of dimensions 1 through 4.}
    \label{fig:logo}
  \end{center}
\end{figure}

\section{Related results}

\cite{park2011strong} obtained the following result.

\begin{theorem}\emph{ (\cite{park2011strong})}\label{pro1-3}
Suppose $n\geq 2$, then $smp(K_n)=n-1$.
\end{theorem}

\cite{cheng2010matching} investigated the matching preclusion number of $AQ_n$ for $n\geq 1$.
\begin{theorem}\emph{(\cite{cheng2010matching})}\label{th1-1}
Let $n\geq 1$. Then $mp(AQ_n) = 2n-1$, that is, $AQ_n$ is maximally matched. If $n\geq 3$, then every optimal
matching preclusion set is trivial, that is, $AQ_n$ is super matched.
\end{theorem}

\cite{cheng2013strong} investigated the strong matching preclusion number of $AQ_n$ for $n\geq 4$.

\begin{theorem}\emph{(\cite{cheng2013strong})}\label{th1-2a}
Let $n\geq 4$. Then $smp(AQ_n)=2n-1$, that is, $AQ_n$ is strongly maximally matched.
\end{theorem}

We remark that the result given by \cite{cheng2013strong} is actually stronger as it also classify all the optimal strong matching preclusion sets.
Note that $AQ_1$ and $AQ_2$ are isomorphic to $K_2$ and $K_4$, respectively, so we acquire $smp(AQ_1)=1$ and $smp(AQ_2)=3$ by Theorem \ref{pro1-3}.
Theorem \ref{th1-2a} can be generalized to include generalized augmented cubes.

\begin{theorem}\emph{(\cite{chang2015strong})}\label{th1-3}
Let $n\geq 4$ and $G\in {\cal GAQ}_n$. Then $smp(G)=2n-1$, that is, $G$ is strongly maximally matched.
\end{theorem}

We remark that Theorem \ref{th1-3} was not explicitly stated by \cite{chang2015strong} but it is implied by Theorem 3.2 in \cite{chang2015strong} and $smp(AQ_4)=7$. In fact, they also classified the optimal strong matching
preclusion sets for a subclass of these generalized augmented cubes.

There is a result for fractional perfect matchings that is analogous to Tutte's Theorem for perfect matchings.

\begin{theorem}\emph{(\cite{tutte1947factorization})}\label{po}
A graph $G$ has a perfect matching if and only if $o(G-S)\leq |S|$ for every set $S\subseteq V(G)$,
where $o(G-S)$ is the number of odd components of $G-S$.
\end{theorem}

\begin{theorem}\emph{(\cite{scheinerman2011fractional})}\label{pi}
A graph $G$ has a fractional perfect matching if and only if $i(G-S)\leq |S|$ for every set $S\subseteq V(G)$,
where $i(G-S)$ is the number of isolated vertices of $G-S$.
\end{theorem}

\cite{liu2017fractional} proved the following result.
\begin{theorem}\emph{(\cite{liu2017fractional})}\label{t14}
Let $n\geq 3$. Then $fsmp(K_n)=n-2$.
\end{theorem}

\section{Main results}
For convenience, we first present some notations, which will be used throughout this
section. If $G\in {\cal GAQ}_n$ for $n\geq 5$, the two subgraphs of $G$ that belong to $\mathcal{GAQ}_{n-1}$ are denoted by $H_0$ and $H_1$. Given $G\in {\cal GAQ}_n$ and $F\subseteq V(G)\cup E(G)$, we denote the subset of $F$ in $H_0$ and $H_1$ by $F^0$ and $F^1$, respectively, and
let $F_V=F\cap V(G)$, $F_E=F\cap E(G)$, $F^i_V=F\cap V(H_i)$,
and $F^i_E=F\cap E(H_i)$, where $i=0,1$.

Our first goal is to find the fractional strong matching preclusion number of generalized augmented cubes. We first claim that if $n\geq 4$ and $G\in {\cal GAQ}_n$, then $fsmp(G)=2n-1$. We start with the following lemma.

\begin{lemma}
\label{lem1}
Let $G$ be generalized augmented cube. Let $(a,b)$ be an edge of $G$, $A$ be the set of neighbors of $a$ and $B$ be the set of neighbors of $b$. Then $A-\{b\}\neq B-\{a\}$.
\end{lemma}
\begin{proof}
 We first show the claim is true for $AQ_4$. If $(a,b)$ is an edge in some one copy of $AQ_4$, it is obvious that $a^x\neq b^x$. Thus, $A-\{b\}\neq B-\{a\}$ for $AQ_4$. Next, we consider that $(a,b)$ is a cross edge or a complement edge of $AQ_4$. Without less generality, we assume $a\in V(AQ^{0}_3)$ and $b\in V(AQ^{1}_3)$. Since $AQ^{i}_3$ is $5$-regular, where $i=0,1$, it follows that $a$ and $b$ have five neighbors in $AQ^0_3$ and $AQ^1_3$, respectively. By definition, we know that $a$ has only one neighbor except $b$ in $AQ^1_3$. Similarly, $b$ has only one neighbor except $a$ in $AQ^0_3$. Thus, $A-\{b\}\neq B-\{a\}$ for $AQ_4$.
 Therefore, the claim is true by the recursive definition of generalized augmented cubes.
\end{proof}

\begin{theorem}\label{Th3.2}
Let $n\geq 5$. If every graph in ${\cal GAQ}_{n-1}$ has fractional strong matching preclusion number $2n-3$, that is, every graph in ${\cal GAQ}_{n-1}$ is fractional strongly maximally matched, then every graph in ${\cal GAQ}_{n}$ has fractional strong matching preclusion number $2n-1$, that is, every graph in ${\cal GAQ}_{n}$ is fractional strongly maximally matched.
\end{theorem}
\begin{proof}
Let $G\in {\cal GAQ}_{n}$. Then $fsmp(G)\leq \delta(G)=2n-1$. Let $F\subseteq V(G)\cup E(G)$ where $|F|\leq 2n-2$.
By definition, $G$ is constructed by using $H_0$ and $H_1$ in ${\cal GAQ}_{n-1}$ together with two edge disjoint
perfect matchings between $V(H_0)$ and $V(H_1)$. Let $v\in V(H_0)$. We denote the edge incident to $v$ from the first set by $(v,v^a)$ and the one from the second set by $(v,v^b)$. Although we do not explicitly define the set of edges in $F$ that are between $H_0$ and $H_1$, the proof will consider these edges.

We want to prove that $G-F$ has a
fractional perfect matching. If $|F_V|$ is even, then $G-F$ has a perfect matching by
Theorem \ref{th1-3}. So we only consider the case that $|F_V|$ is odd. We may assume that $|F^0|\geq |F^1|$.

{\bf Case 1.} $|F^0|=2n-2$. Then $F=F^0$. Since $|F_V|$ is odd, $|F^0_V|\geq 1$. Let $v\in F^0_V$. Since $2n-2$ is even and
$|F_V|=|F^0_V|$ is odd, $F^0$ contains an edge $(w,s)$. Let $F^{00}=F^0-\{v, (w,s)\}$. So $|F^{00}|=2n-4$.
Since $H_0-F^{00}$ has an even number of vertices, there exists a perfect matching $M$ by Theorem \ref{th1-3}.
We first assume that $(w,s)\in M$. Now $(v,y)\in M$ for some $y$. Clearly $y\not\in\{w,s\}$. Now
$H_1-\{y^a, w^a, s^a\}$ has
a fractional matching $f_1$ by assumption as $2n-3>3$ for $n\geq 4$. Let $M'=M-\{(y,v),(w,s)\}$. Then it is clear that $M'\cup\{(y,y^a),(w,w^a),(s,s^a)\}$ and $f_1$ induce a fractional
prefect matching of $G-F$. The argument for the case when $(w,s)\not\in M$ is easier. Consider $(v,y)\in M$ for some $y$.  Now
$H_1-\{y^a\}$ has
a fractional matching $f_1$ by assumption as $2n-3>1$. Let $M'=M-\{(y,v))\}$. Then it is clear that $M'\cup\{(y,y^a)\}$ and $f_1$ induce a fractional
prefect matching of $G-F$.

{\bf Case 2.} $|F^0|=2n-3$. Then $|F^1|\leq 1$. We consider two subcases.

{\em Subcase 2.1.} $F^0$ contains an odd number of vertices. Then let $v\in F^0_V$. Let $F^{00}=F^0-\{v\}$. So $|F^{00}|=2n-4$. Since $H_0-F^{00}$ has an even number of vertices, there exists a perfect matching $M$ by Theorem \ref{th1-3}.
Now $(v,y)\in M$ for some $y$.
Since $|F-F^0|\leq 1$, at least one of $(y,y^a)$ and $(y,y^b)$ is in $G-F$. We may assume that it is $(y,y^a)$.
$H_1-F^1-\{y^a\}$ has
a fractional matching $f_1$ by assumption as $2n-3>2$. Let $M'=M-\{v\}$. Then it is clear that $M'\cup\{(y,y^a)\}$ and $f_1$ induce a fractional
prefect matching of $G-F$.

{\em Subcase 2.2.} $F^0$ contains an even number of vertices. If $F^0$ contains an edge $(u,z)$ then we set $F^{00}=F^0-\{(u,z)\}$. So $|F^{00}|=2n-4$. Since $H_0-F^{00}$ has an even number of vertices, there exists a perfect matching $M$ by Theorem \ref{th1-3}. If $(u,z)\not\in M$, then apply assumption to obtain a fractional perfect matching $f_1$ for $H_1-F^1$. Then it is clear that $M$ and $f_1$ induce a fractional
prefect matching of $G-F$. Now suppose $(u,z)\in M$. Then consider the edges $(u,u^a),(z,z^a),(u,u^b),(z,z^b)$. If they contain two independent edges that are in $G-F$, then we can apply the usual argument to obtain a desired fractional
prefect matching of $G-F$. So assume that we cannot find two independent edges from them. Since $|F-F^0|\leq 1$, we can conclude that $u^a=z^b$, $u^b=z^a$ and one of $u^a$ and $u^b$
is in $F$. But this can only occur for one such pair.

Thus we simply consider a different edge in $F^0$ unless all the remaining elements of $F$ are vertices. Suppose that it contains a vertex $w$ that is not adjacent to both $u$ and $z$.
Then pick another vertex $s$ in $F$. Let $F^{00}=F^0-\{w,s\}$ if $w$ is adjacent to neither $u$ nor $z$. If $w$ is adjacent to one of them, say $u$, then let $F^{00}=(F^0-\{w,s\})\cup\{(w,u)\}$. Thus $|F^{00}|\leq 2n-3-2+1=2n-4$.
Since $H_0-F^{00}$ has an even number of vertices, there exists a perfect matching $M$ by Theorem \ref{th1-3}. Consider $(w,y),(s,v)\in M$. By choice of $w$ and construction of $F^{00}$, $y\notin\{u,z\}$. Therefore $(y,y^a),(v,v^a),(y,y^b),(v,v^b)$ contain two independent edges that are in $G-F$ as $\{y,v\}\neq\{u,z\}$.

Thus we have identified $F$. $F$ consists of $(u,z)$ together with $2n-4$ vertices, each is adjacent to both $u$ and $z$. This is a contradiction by Lemma \ref{lem1}.

{\bf Case 3.} $|F^0|\leq 2n-4$.
Then $|F^1|\leq 2n-4$. By assumption, $H_0-F^0$ and $H_1-F^1$ have fractional perfect matchings $f_0$ and $f_1$, respectively, which induce a fractional
prefect matching of $G-F$.
\end{proof}

Thus it follows from Theorem \ref{Th3.2} that if we can show that $AQ_4$ is fractional strongly maximally matched, then every generalized augmented cube is fractional strongly maximally matched.
We now turn our attention to the classification of optimal fractional strong matching preclusion sets of graphs in ${\cal GAQ}_n$. We start with the following lemma.

\begin{lemma}
\label{lem1a}
Let $G$ be a generalized augmented cube.
\begin{itemize}
\item Let $(a,b)$ be an edge of $G$, $A$ be the set of neighbors of $a$ and $B$ be the set of neighbors of $b$. Then $|(A-\{b\})\setminus (B-\{a\})|\geq 2$.
\item Let $a$ and $b$ be nonadjacent vertices of $G$, $A$ be the set of neighbors of $a$ and $B$ be the set of neighbors of $b$. Then $|A \setminus B|\geq 2$.
\end{itemize}
\end{lemma}
\begin{proof}
 We first show the claim is true for $AQ_4$. Consider any two distinct vertices $a$ and $b$ of $AQ_4$. If $a$ and $b$ are in different copy of $AQ_4$, without less generality, we assume $a\in V(AQ^{0}_3)$ and $b\in V(AQ^{1}_3)$. Since $AQ^{i}_3$ is $5$-regular, where $i=0,1$, it follows that $a$ and $b$ have five neighbors in $AQ^0_3$ and $AQ^1_3$, respectively. It implies that there exist at least three neighbors of $a$ in $AQ^0_3$ such that they are not adjacent to $b$. Similarly, there exist at least three neighbors of $b$ in $AQ^1_3$ such that they are not adjacent to $a$. Thus, $|(A-\{b\})\setminus (B-\{a\})|\geq 2$ or $|A \setminus B|\geq 2$. Next, we consider that $a$ and $b$ are in some copy of $AQ_4$. Without less generality, we assume $a$ and $b$ are in $AQ^0_3$. It is clear that $a^x\neq b^x$. If we can find at least a pair of distinct neighbors of $a$ and $b$ in $AQ^0_3$, the claim is true. If $a$ and $b$ are in different copy of $AQ_2$, we can find a pair of distinct neighbors of $a$ and $b$ in different copy of $AQ_2$ as the copy of $AQ_2$ is $3$-regular. If $a$ and $b$ are in same one copy of $AQ_2$, then the neighbors $a$ and $b$ in cross edges are distinct. Therefore, the claim is true by the recursive definition of generalized augmented cubes. We note that one can also verify the statement for $AQ_4$
easily via a computer, and we have performed this verification.
\end{proof}
We note that Lemma \ref{lem1a} implies the following: If $G\in {\cal GAQ}_{n}$ (where $n\geq 4$), then $G$ does not contain a $K_{2,2n}$ as a subgraph. This remark will be useful later. Then we have the following result.

We first note the following result.

\begin{theorem}{\upshape \cite{cheng2013strong}}\label{th1-3aa}
Let $n\geq 4$. Then $AQ_n$ is strongly super matched.
\end{theorem}

In fact, we will only need a special case of it.

\begin{corollary}\label{th1-3a}
Let $n\geq 4$. Let $F\subseteq V(AQ_n)\cup E(AQ_n)$ be an optimal strong matching preclusion set with an even number of vertices. Then $F$ is trivial.
\end{corollary}

We will call a graph $G$ \emph{even strongly super matched} if it is strongly maximally matched and every optimal strong matching preclusion set with an even number of vertices is trivial. So Corollary \ref{th1-3a} says $AQ_n$ is even strongly super matched if $n\geq 4$.
We are now ready to prove the following result.

\begin{theorem}\label{Th3.3}
Let $n\geq 5$. Suppose
\begin{enumerate}
\item every graph in ${\cal GAQ}_{n-1}$ is even strongly super matched, and
\item every graph in ${\cal GAQ}_{n-1}$ is fractional strongly super matched.
\end{enumerate}
Then every graph in ${\cal GAQ}_{n}$ is fractional strongly super matched.
\end{theorem}
\begin{proof}
Let $G\in {\cal GAQ}_{n}$. Let $F\subseteq V(G)\cup E(G)$ where $|F|=2n-1$ and $F$ is optimal.
We follow the same notation as in the proof of Theorem \ref{Th3.2}.

We want to prove that $G-F$ either has a
fractional perfect matching or $F$ is trivial. If $|F_V|$ is even, then $G-F$ either has a perfect matching or $F$ is trivial by the assumption that every graph in ${\cal GAQ}_{n-1}$ is even strongly super matched. So we only consider the case that $|F_V|$ is odd. We may assume that $|F^0|\geq |F^1|$.

{\bf Case 1.} $|F^0|=2n-1$. Then $F=F^0$. Since $|F_V|$ is odd, $|F^0_V|\geq 1$. Let $v\in F^0_V$. We consider two subcases.

{\em Subcase 1.1.} $F^0$ contains an edge $(w,s)$. Let $F^{00}=F^0-\{v, (w,s)\}$. So $|F^{00}|=2n-3$.
Since $H_0-F^{00}$ has an even number of vertices, it either has a perfect matching $M$ or $F^{00}$ is trivial by assumption 1. If it has a perfect matching $M$, then the argument of Case 1 in the proof of Theorem~\ref{Th3.2} applies. Thus, we may assume that $F^{00}$ is trivial and that it is induced by a vertex, say,
$\hat{u}$. If $F^{00}$ contains an edge $(w',s')$, replace $(w,s)$ by $(w',s')$ to obtain $F^{000}$, and repeat the argument. If $F^{00}$ contains vertices only, replace $v$ by one of them to obtain $F^{000}$, and repeat the argument. We only have to consider the case that $F^{000}$ is trivial and that it is induced by a vertex, say, $u'$. Since $|F^{00}\setminus F^{000}|=1$, this violates Lemma \ref{lem1a}.

{\em Subcase 1.2.} $F^0$ contains all vertices. Since $2n-1\geq 3$, pick two additional vertices $u$ and $z$ in $F^0$. Let $F^{00}=F^0-\{v,u,z\}$.
Since $H_0-F^{00}$ has an even number of vertices and $|F^{00}|=2n-4$, there exists a perfect matching $M$ by  Theorem \ref{th1-3}.
Consider $(v,y), (u,w), (z,s)\in M$ for some $y,w,s$. Now
$H_1-\{y^a, w^a, s^a\}$ has
a fractional matching $f_1$ by assumption 2 as $2n-3>3$ for $n\geq 4$. Let $M'=M-\{v,u,z\}$. Then it is clear that $M'\cup\{(y,y^a),(w,w^a),(s,s^a)\}$ and $f_1$ induce a fractional
prefect matching of $G-F$.

{\bf Case 2.} $|F^0|=2n-2$. Then $|F^1|\leq 1$. We consider two subcases.

{\em Subcase 2.1.} $F^0$ contains an odd number of vertices. Then let $v\in F^0_V$. Let $F^{00}=F^0-\{v\}$. So $|F^{00}|=2n-3$. Since $H_0-F^{00}$ has an even number of vertices, it either has a perfect matching $M$ or $F^{00}$ is trivial by assumption 1. If it has a perfect matching $M$, then the argument of Subcase 2.1 in the proof of Theorem~\ref{Th3.2} applies. Thus, we may assume that $F^{00}$ is trivial and that it is induced by a vertex, say,
$\hat{u}$. Since $|F^0_V|$ is odd, we can pick $v'\in F^0_V$, let $F^{000}=F^0-\{v'\}$,
and repeat the argument. We only have to consider the case that $F^{000}$ is trivial and that it is induced by a vertex, say, $u'$. Since $|F^{00}\setminus F^{000}|=1$, this violates Lemma \ref{lem1a}.

{\em Subcase 2.2.} $F^0$ contains an even number of vertices. We consider two subcases.

{\em Subcase 2.2.1} $H_0-F^0$ contains an isolated vertex $v$. We may assume that
$(v,v^a)$ is in $G-F$. Either $|F^0_E|\geq 2$ or $|F^0_E|=0$.
We first suppose $|F^0_E|\geq 2$. Then there is $(u,v) \in F^0_E$ such
that $(u,u^a)$ is in $G-F$. Then let $F^{00}=F^0-\{(u,v)\}$. So $|F^{00}|=2n-3$. It is not difficult to check that it follows from Lemma \ref{lem1a} that
$H_0-F^{00}$ has no isolated vertices. Since $H_0-F^{00}$ has an even number of vertices, it has a perfect matching $M$ by assumption 1. Now $(u,v)\in M$. Since $(v,v^a)$ and $(u,u^a)$ are in $G-F$ and $H_1-(F^{1}\cup\{v^a,u^a\})$ has a fractional perfect matching $f_1$ by assumption 2, it is clear
$(M-\{(u,v)\})\cup\{(v,v^a), (u,u^a)\}$ and $f_1$ induce a fractional perfect matching of $G-F$. We now assume that $|F^0_E|=0$. Then there is a vertex $u$ adjacent to $v$ such that
$(u,u^a)$ is in $G-F$. Then let $F^{00}=F^0-\{u\}$. So $|F^{00}|=2n-3$. It is not difficult to check that it follows from Lemma \ref{lem1a} that
$H_0-F^{00}$ has no isolated vertices. Since $H_0-F^{00}$ has an odd number of vertices, it has an almost perfect matching $M$ missing $w$ by assumption 1. Since $|F-F^0|\leq 1$, there exists at least one of $w^a$ and $w^b$ such that it is not in $F-F^0$, so we can assume that $w^a\notin F-F^0$. Now $(u,v)\in M$. Since $(v,v^a)$ and $(w,w^a)$ are in $G-F$ and $H_1-(F^{1}\cup\{v^a,w^a\})$ has a fractional perfect matching $f_1$ by assumption 2, it is clear
$(M-\{u\})\cup\{(v,v^a), (w,w^a)\}$ and $f_1$ induce a fractional perfect matching of $G-F$.

{\em Subcase 2.2.2} $H_0-F^0$ has no isolated vertices. This implies $H_0-F'$ has no isolated vertices if $F'\subseteq F^0$.
We first suppose $F^0$ contains edges. Then it must contain at least two. Let $(u,z)$ be such an edge.
Let $F^{00}=F^0-\{(u,z)\}$. So $|F^{00}|=2n-3$. Since $H_0-F^{00}$ has no isolated vertices and it has an even number of vertices, it has a perfect matching $M$ by assumption 1.
If $(u,z)\not\in M$, then apply assumption 2 to obtain a fractional perfect matching $f_1$ for $H_1-F^1$. Then it is clear that $M$ and $f_1$ induce a fractional
prefect matching of $G-F$. Now suppose $(u,z)\in M$. Then consider the edges $(u,u^a),(z,z^a),(u,u^b),(z,z^b)$. If they contain two independent edges that are in $G-F$, without loss of generality, assume that $(u,u^a),(z,z^a)$ are two independent edges in $G-F$. By assumption 2, $H_1-(F^{1}\cup\{u^a,z^a\})$ has a fractional perfect matching $f_1$. Therefore $(M-\{(u,z)\})\cup\{(u,u^a), (z,z^a)\}$ and $f_1$ induce a fractional perfect matching of $G-F$. So now assume that we cannot find two independent edges from them. Since $|F-F^0|\leq 1$, we can conclude that $u^a=z^b$, $u^b=z^a$ and one of $u^a$ and $u^b$
is in $F$. But this can only occur for one such pair. Thus we simply pick a different edge.

Thus we may assume that $F^0=F^0_V$. Pick two vertices $v,y\in F^0_V$. Consider $F^{00}=F^0-\{v,y\}$. So $|F^{00}|=2n-4$. Then $H_0-F^{00}$
has a perfect matching $M$ by Theorem \ref{th1-2a}. If $(v,y)\in M$, then it is easy to find a fractional perfect matching of $G-F$ with the above stated method. Thus we assume
$(v,u),(y,z)\in M$ for some other vertices $u$ and $z$. If $u$ and $z$ are adjacent, then it is also easy. Thus we assume that $u$ and $z$ are not adjacent. Now consider the edges $(u,u^a),(z,z^a),(u,u^b),(z,z^b)$. If they contain two independent edges that are in $G-F$, without loss of generality, assume that $(u,u^a),(z,z^a)$ are two independent edges in $G-F$. By assumption 2, $H_1-(F^{1}\cup\{u^a,z^a\})$ has a fractional perfect matching $f_1$. Therefore $(M-\{(v,u),(y,z)\})\cup\{(u,u^a), (z,z^a)\}$ and $f_1$ induce a fractional perfect matching of $G-F$. So assume that we cannot find two independent edges from them. Since $|F-F^0|\leq 1$, we can conclude that $u^a=z^b$, $u^b=z^a$ and one of $u^a=z^b$ and $u^b=z^a$
is in $F$. Without loss of generality, we may assume that it is $z^a$. But this can only occur for one such pair.

We consider $F^{00}=(F^0-\{v_1,v_2,v_3,v_4\})\cup \{u,u'\}$ where $(u,u')$ is in $H_0-F^0$. Note that by assumption $u'\neq z$. Then $|F^{00}|=2n-4$. Thus $H_0-F^{00}$
has a perfect matching $M$ by Theorem \ref{th1-2a}. Thus we have edges $(v_1,v_1'),(v_2,v_2'),(v_3,v_3'),(v_4,v_4')\in M$. In the worst case, all of these vertices are distinct. By the construction of $M$, none of the
$v_i'$ is $u$, though one of them may be $z$. Therefore  $(v_1',v_1'^b),(v_2',v_2'^b),(v_3',v_3'^b),$ and $(v_4',v_4'^b)$ are independent edges in $G-F$. Then the usual argument gives a required fractional perfect matching of $G-F$. However, this requires we consider
$F^{11}=F^1\cup\{v_1'^b,v_2'^b,v_3'^b,v_4'^b\}$ and $|F^{11}|\leq 5$. Since $2n-4\geq 5$ for $n\geq 5$, we are done. (If the $v_i$ and $v_i'$ are not all distinct, i.e. we have $v_i' = v_j$ for some $i \neq j$, then we need not consider the vertices $v_i'^b$ and $v_j'^b$. This means we have $|F^{11}| < 5$, and we obtain a fractional perfect matching in the same way.)

{\bf Case 3.} $|F^0|=2n-3$. By assumption 2, $H_0-F^0$ either has a fractional perfect matching $f_0$ or that $F^0$ is trivial. In the first case,
$H_1-F^1$ has a fractional perfect matching $f_1$. Thus $f_0$ and $f_1$ induce a fractional
prefect matching of $G-F$. Therefore, we assume that $F^0$ is trivial and it is induced by the vertex $v$. If $F$ is trivial with respect to $G$, then
we are done. Thus we may assume that $(v,v^a)$ is in $G-F$. We consider two subcases.

{\em Subcase 3.1.} $|F^0_V|$ is odd. We first suppose $|F^0_V|\geq 3$. Then we may find $u\in F^0_V$ such that $(u,u^a)$ is in $G-F$. Let
$F^{00}=F^0-\{u\}$. So $|F^{00}|=2n-4$. Then $H_0-F^{00}$ has an even number of vertices, and it has a perfect matching $M$ by Theorem \ref{th1-3}. Moreover, $(u,v)\in M$. Since $(v,v^a)$ and $(u,u^a)$ are in $G-F$, $H_1-(F^{1}\cup\{v^a,u^a\})$ has a fractional perfect matching $f_1$ by assumption 2, it is clear
$(M-\{(u,v)\})\cup\{(v,v^a), (u,u^a)\}$ and $f_1$ induce a fractional perfect matching of $G-F$.

We now suppose $|F^0_V|=1$ and $F^0_V=\{u\}$. Then $H_0-F^0-\{v\}=H_0-\{v,u\}$. By Theorem \ref{th1-3}, $H_0-\{v,u\}$ has a perfect matching $M$. Since $(v,v^a)$ is in $G-F$ and $H_1-(F^{1}\cup\{v^a\})$ has a fractional perfect matching $f_1$ by assumption 2, it follows that $M\cup\{(v,v^a)\}$ and $f_1$ induce a fractional perfect matching of $G-F$.

{\em Subcase 3.2.} $|F^0_V|$ is even. So $|F^0_E|$ is odd. We first suppose $|F^0_E|\geq 3$. Then there is $(u,v)\in F^0_E$ such
that $(u,u^a)$ is in $G-F$. Then let $F^{00}=F^0-\{(u,v)\}$. So $|F^{00}|=2n-4$. Since $H_0-F^{00}$ has an even number of vertices, it has a perfect matching $M$ by Theorem \ref{th1-3}.
Moreover, $(u,v)\in M$. Since $(v,v^a)$ and $(u,u^a)$ are in $G-F$ and $H_1-(F^{1}\cup\{v^a, u^a\})$ has a fractional perfect matching $f_1$ by assumption 2, it follows that $(M-\{(u,v)\})\cup\{(v,v^a), (u,u^a)\}$ and $f_1$ induce a fractional perfect matching of $G-F$.

We now suppose $|F^0_E|=1$. Pick any vertex $w$ in $H_0-F^0$ such that $(w,w^a)$ is in $G-F$. Since $|F^0_V|\geq 3$, we may let $F^{00}=(F^0-\{u\})\cup\{w\}$, where $u\in F^0$. So $|F^{00}|=2n-3$. It follows from Lemma \ref{lem1a} that
$G-F^{00}$ has no isolated vertices. Since $H_0-F^{00}$ has an even number of vertices, it has a perfect matching $M$ by assumption 1.
Moreover, $(u,v)\in M$. Since $(v,v^a),(w,w^a)$ are in $G-F$ and $H_1-(F^{1}\cup\{v^a, w^a\})$ has a fractional perfect matching $f_1$ by assumption 2, it follows that $(M-\{u\})\cup\{(v,v^a), (w,w^a)\}$ and $f_1$ induce a fractional perfect matching of $G-F$.

{\bf Case 4.} $|F^0|\leq 2n-4$.
Then $|F^1|\leq 2n-4$. By assumption 2, $H_0-F^0$ and $H_1-F^1$ have fractional perfect matchings $f_0$ and $f_1$, respectively, which induce a fractional
prefect matching of $G-F$.
\end{proof}

Thus it follows from Theorem \ref{Th3.3} that if we can show that $AQ_4$ is fractional strongly super matched, then every $AQ_n$ is fractional strongly super matched for $n\geq 5$
since assumption 1 in Theorem \ref{Th3.3} is given by Corollary \ref{th1-3a}. The question is how about the generalized augmented cubes? We consider the following subclass which
we call \emph{restricted generalized a-cubes}.
We define the set ${\cal RGAQ}_4=\{AQ_4\}$. For $n\geq 5$, ${\cal RGAQ}_n$ consists of all graphs that can be obtained in the following way: Let $G_1,G_2\in {\cal GAQ}_{n-1}$, where $G_1=(V_1,E_1)$ and $G_2=(V_2,E_2)$ may not be distinct; construct the graph $(V_1\cup V_2,E_1\cup E_2\cup M_1\cup M_2)$ where $M_1$ and $M_2$
are edge disjoint perfect matchings between $V_1$ and $V_2$ where $M_1$ and $M_2$ induce neither 4-cycles nor 6-cycles. The reason we use the term
restricted generalized a-cubes rather than restricted generalized augmented cubes because augmented cubes do not belong to this class of graph as the cross edges and
the complement edges will induce 4-cycles. The reason we consider this class is because we can utilize a result of \cite{chang2015strong}. To use it, we also need to show that
if $G\in {\cal RGAQ}_n$, then $G$ does not contain $K_{2,2n-2}$, which is implied by Lemma \ref{lem1a}, as noted ealier. Using this result (Theorem 3.5 in \cite{chang2015strong}) and
$AQ_4$ is strongly super matched (Theorem \ref{th1-3aa}), we have the following result and its immediate corollary.

\begin{theorem}{\upshape \cite{cheng2013strong,chang2015strong}}
Every restricted generalized a-cube is strongly super matched.
\end{theorem}

\begin{corollary}\label{th1-3b}
Every restricted generalized a-cube is even strongly super matched.
\end{corollary}

Thus it follows from Theorem \ref{Th3.3} that if we can show that $AQ_4$ is fractional strongly super matched, then every restricted generalized a-cube is fractional strongly super matched
since assumption 1 in Theorem \ref{Th3.3} is given by Corollary \ref{th1-3b}. Finally we note that if a graph is fractional strongly super matched, then it is fractional super matched. Thus
it is not necessary to consider the second concept in this paper. We now present the main results of this paper.

\begin{theorem}
\label{maintheorem1}
Let $n\geq 4$. Then $AQ_n$ is fractional strongly maximally matched and fractional strongly super matched.
\end{theorem}

\begin{theorem}
\label{maintheorem2}
Every generalized augmented cube is fractional strongly maximally matched and every restricted generalized a-cube is fractional strongly super matched.
\end{theorem}

We will complete the proof of these results in the next section by showing that $AQ_4$ is fractional strongly maximally matched and fractional strongly super matched.
We remark that \cite{cheng2013strong} showed that $AQ_4$ is strongly maximally matched and strongly super matched via computer verification. We could do the same here as it is not more difficult. Determining whether $AQ_4-F$ has a fractional perfect matching is just as simple as determining $AQ_4-F$ has a perfect matching or an almost prefect
matching, as the first problem can be solved by solving a simple linear program and the second problem can be solved by an efficient matching algorithm.

One may wonder whether Theorems \ref{maintheorem1} and \ref{maintheorem2} can be strengthened from all restricted generalized a-cubes and augmented cubes to all generalized augmented cubes in terms of fractional strongly super matchedness. We did not investigate this. However, we will point out in the proof of Theorem 3.5 in \cite{chang2015strong},  the condition of no $4$-cycles and no $6$-cycles is important.

\section{The base case}
In this section, we prove the following lemma, which is the base case, of our argument via a computational approach.

\begin{lemma}
\label{mainlemma}
$AQ_4$ is fractional strongly maximally matched and fractional strongly super matched.
\end{lemma}
\begin{proof}
	This result was verified by a computer program written in the Python language, using the NetworkX package (https://networkx.github.io/) to represent the structure of the graph and the SciPy package (https://www.scipy.org/) to compute fractional perfect matchings. The program verified that for any $7$-element fault-set $F$, either $F$ is trivial or $AQ_4-F$ has a fractional perfect matching. We note that this condition, in addition to verifying that $AQ_4$ is fractional strongly super matched, is sufficient to verify that $AQ_4$ is fractional strongly maximally matched; this follows from the fact that any fault set $F$ in $AQ_4$ with $6$ or fewer elements may be extended to a non-trivial $7$-element fault set $F'$ by including additional edges, and that if $AQ_4 - F$ does not have a fractional perfect matching then $AQ_4 - F'$ also does not have a fractional perfect matching. We may reduce the number of cases that need to be checked by noting that Theorem \ref{th1-1} implies that $F$ must contain at least one vertex, and furthermore that the vertex-transitivity of $AQ_4$ implies that one vertex of $F$ may be fixed. The additional simplifying assumption was made that no fault edge is incident to any fault vertex. The computer verification took about two days on a typical desktop computer.
\end{proof}

Originally we intended to prove Lemma~\ref{mainlemma} theoretically. Indeed, we have a long proof with many cases to establish that $AQ_4$ is fractional strongly maximally matched. The super version will be
even more involved. Thus we decided a computational approach is cleaner. Moreover, it demonstrates how even a straightforward implementation is useful. We could reduce the number of cases to check by further applying properties of $AQ_4$ but we decided it is not necessary to increase the complexity of the program. Indeed the program is short. The program is given in the Appendix.

\section{Conclusion}
The fractional strong matching preclusion problem was introduced in \cite{liu2017fractional}. In this paper, we explore this parameter for a large class of cube-type interconnection networks including augmented cubes. It would be interesting to consider this parameter in future projects for competitors of cube-like networks such as $(n,k)$-star graphs and arrangement graph. Another possible direction is to consider this parameter for general
products of networks.

\acknowledgements
\label{sec:ack}
We would like to thank the anonymous referees for a number of helpful comments and suggestions.

\nocite{*}
\bibliographystyle{abbrvnat}
\bibliography{fmpaug}
\label{sec:biblio}

\appendix
\section {Program for Lemma~\ref{mainlemma}}

\begin{verbatim}
#strong fractional matching preclusion problem for the augmented cube AQ_4
#16 vertices, 7-regular
import networkx as nx
import itertools as itr
import scipy
import numpy
from timeit import default_timer as timer

#defines AQ_4 using the networkx representation
def aq4():
    A = nx.Graph()
    c = ['0','1']
    for i in range(3):
        d = [s+'1' for s in c]
        c = [s+'0' for s in c]
        c = c + d
    p = itr.combinations(c,2)
    for i,j in p:
        if is_adj_aq4(i,j):
            A.add_edge(i,j)
    return A

#function to determine adjacency between vertices in the augmented cube
def is_adj_aq4(i,j):
    for n in range(len(i)):
        if ((i[n] == string_complement(j[n])) and (i[0:n] == j[0:n])
        and (i[n+1:] == j[n+1:]))
        or ((i[:n] == j[:n])
        and (i[n+1:] == string_complement(j[n+1:])) and (i != j)):
            return True
    return False

#auxiliary function used in the adjacency check
def string_complement(i):
    s = ''
    for char in i:
        if char == '0':
            s = s + '1'
        else:
            s = s + '0'
    return s

#determines if a matching preclusion set is basic or not
def is_basic(G):
    if min(G.degree().values()) == 0:
        return True
    else:
        return False


#code to determine the fsmp sets of AQ_4 with n vertices removed
#here we consider n>0
#n=0 requires different logic, and has been checked previously
def fpm_aq4(n):
    G = aq4()
    #by vertex transitivity, we can remove one vertex WLOG
    G.remove_node('0000')
    count = 0
    tcount = 0
    start = timer()

	#choose additional vertices to be removed, for a total of n
    for r in itr.combinations(G.nodes(),n-1):
        H = G.copy()
        H.remove_nodes_from(r)

      	#impose a fixed order on the edges to construct the LP matrix
        edge_order = {}
        for e in range(len(H.edges())):
            edge_order[H.edges()[e]] = e
        M = [[] for i in range(len(H.nodes()))]
        for v in range(len(H.nodes())):
            for e in range(len(H.edges())):
                if H.nodes()[v] in H.edges()[e]:
                    M[v].append(1)
                else:
                    M[v].append(0)
        b = [1 for i in range(len(H.nodes()))]
        c = [-1 for i in range(len(H.edges()) - 7 + n)]
		
		#choose edges to remove, and remove corresponding columns
        for l in itr.combinations(H.edges(),7-n):
            ll = [edge_order[i] for i in l]
			#solve the fractional matching LP
            res = scipy.optimize.linprog(c = c, A_ub = numpy.delete(M,ll,axis=1), b_ub = b)
            #if no FPM exists, check the obstruction set
            #if basic, do nothing
            #if not basic, record
            if res.fun > -1*.5*len(H.nodes()):
                tcount = tcount + 1
                print('fsmp set found: '+str(tcount)+' times')
                K = H.copy()
                K.remove_edges_from(l)
                if not is_basic(K):
                    print('nontrivial fsmp set found')
                    with open('fsmp-log.txt','w') as f:
                        f.writelines([str(r),str(l),str(K.degree()[min(K.degree())]),
                        str(len(K.nodes())),str(len(K.edges())),str(K.degree()),
                        str(res),str(M),str(b),str(c)])
                        f.close()
            count = count + 1
            if count % 1000 == 0:
                print(str(count) + '\n')
            if count == 10000:
                end = timer()
                print('test done:' + str(end - start) + ' ' + 'seconds elapsed')
    end = timer()
    print('all strong fmp sets of AQ4 with ' + str(n) + ' vertices are basic')
    print(str(end - start) + ' ' + 'seconds elapsed')
    return True

if __name__ == "__main__":
    if fpm_aq4(7) == True:
        input()
\end{verbatim}

\end{document}